\documentclass{amsart}

\usepackage{latexsym}		  
\usepackage{hyperref}
\usepackage{amsfonts}
\usepackage{amsmath}
\usepackage{amssymb}

\newtheorem{theorem}{Theorem}[section]
\newtheorem{lemma}[theorem]{Lemma}
\newtheorem{corollary}[theorem]{Corollary}
\newtheorem{proposition}[theorem]{Proposition}

\numberwithin{equation}{section}

\newcount\quantno
\everydisplay{\quantno=0}\everycr{\quantno=0}
\def\quant{\advance\quantno by1
                         \ifnum\quantno=1\qquad\else\quad\fi\forall }

\renewcommand\Re{\operatorname{\mathrm{Re}}}

\newcommand\rest[1]{\kern-.1em
             \lower.5ex\hbox{$\scriptstyle #1$}\kern.05em}

\newcommand\vett[2] {{#1}_1,\ldots,{#1}_{#2}} 
\newcommand\set[1]{{\left\{#1\right\}}}

\renewcommand\mod[1]{\left\vert{#1}\right\vert}

\newcommand\norm[2]{{\left\Vert{#1}\right\Vert_{#2}}}

\newcommand\wrt{\,\text{\rm d}}

\newcommand\diag{\,\text{\rm diag}}

\newcommand\sgr[1]{(\e^{t#1})_{t\ge 0}}
\newcommand\pint[1]{{[\hskip-2pt[#1]\hskip-2pt]}} 

\newcommand\BN{\mathbb{N}}
\newcommand\BR{\mathbb{R}}

\newcommand\BT{\mathbb{T}}

\newcommand\cG{\mathcal{G}}
\newcommand\cH{\mathcal{H}}

\newcommand\cL{\mathcal{L}}

\newcommand\cP{\mathcal{P}}  
\newcommand\cQ{\mathcal{Q}}
\newcommand\cR{\mathcal{R}}  
\newcommand\cS{\mathcal{S}}  
\newcommand\cT{\mathcal{T}}  
  
\newcommand\cW{\mathcal{W}}

\newcommand\tr{\mathop{\rm tr}}

\newcommand\ga{\gamma_\infty}

\newcommand\loc {\hbox{{\tiny loc}}}
\newcommand\glob {\hbox{{\tiny glob}}}
\newcommand\1{{\bf 1}} 

\newcommand\funnyk{k\hbox to 0pt{\hss\phantom{g}}}

\newcommand\rmi{\hbox{\rm (i)}}
\newcommand\rmii{\hbox{\rm (ii)}}
\newcommand\rmiii{\hbox{\rm (iii)}}
\newcommand\rmiv{\hbox{\rm (iv)}}

\newcommand\e{{\mathrm e}}

\begin{document}

\author
{Giancarlo Mauceri,  Luana Noselli}


\title{The maximal operator associated to a non-symmetric Ornstein-Uhlenbeck semigroup}
\subjclass[2000]{42B25, 47D03}
\keywords{Ornstein-Uhlenbeck semigroup, maximal operator, weak type.}
\thanks{The authors have received support  by the Italian MIUR-PRIN 2005 project \lq\lq Harmonic Analysis\rq\rq  and by the EU IHP 2002-2006 project \lq\lq HARP\rq\rq}

\title[non-symmetric Ornstein-Uhlenbeck semigroup]{ The maximal operator associated to a non-symmetric Ornstein-Uhlenbeck semigroup}

\maketitle

\begin{abstract}  Let $(\cH_t)_{t\ge 0}$ be the Ornstein-Uhlenbeck semigroup on $\BR^d$ with covariance matrix $I$ and drift matrix $\lambda(R-I)$, where $\lambda>0$ and  $R$ is a skew-adjoint matrix and denote by $\gamma_\infty$  the invariant measure for  $(\cH_t)_{t\ge 0}$. Semigroups of this form are the basic building blocks of Ornstein-Uhlenbeck semigroups which are normal on $L^2(\gamma_\infty)$. We prove that if the matrix $R$ generates a one-parameter group of periodic rotations then the maximal operator $\cH_*f(x)=\sup_{t\ge o}\mod{\cH_tf(x)}$ is of weak type $1$ with respect to the invariant measure $\gamma_\infty$. We also prove that the maximal operator associated to an arbitrary normal Ornstein-Uhlenbeck semigroup is bounded on $L^p(\gamma_\infty)$ if and only if $1<p\le \infty$.

\end{abstract}

\section{Introduction} \label{s:Introduction}
Let $Q$ be a real, symmetric, positive definite $d\times d$-matrix and let $B$ be  a nonzero real $d\times d$-matrix whose eigenvalues have negative real part. Then for every $t\in(0,\infty]$ we can define the family of Gaussian measures $\gamma_t$ on $\BR^d$ with mean zero and covariance operators
\begin{equation}\label{Qt}
Q_{t}=\int_{0}^{t}\e^{sB}Q\e^{sB^{\ast}}\,\wrt s,\quad t\in(0,\infty],
\end{equation}
i.e. the measures
$$
\wrt\gamma_t(x)=(2\pi)^{-d/2} (\det Q_t)^{-1/2}\,\e^{-\frac{1}{2}\langle Q_t^{-1} x,x\rangle}\wrt\lambda(x) \qquad \forall t\in(0,\infty].
$$  The Ornstein-Uhlenbeck semigroup is the family of operators  $(\cH^{Q,B}_t)_{t\geq0}$  defined by
\begin{equation}\label{Ht}
\cH^{Q,B}_t f(x)=\int_{\BR^{d}}f(\e^{tB}x-y)\,\wrt\gamma_t(y)
\end{equation}
on the space $C_b(\BR^d)$ of bounded continuous functions.
The matrices $Q$ and $B$ are called the {\it covariance} and the {\it drift} matrix, respectively.\par
 It is well known that  $\gamma_\infty$ is the unique invariant measure for $\cH^{Q,B}_t$ and that $(\cH^{Q,B}_t)_{t\geq0}$ is a diffusion semigroup on$(\BR^d,\gamma_\infty)$ (see for instance \cite{DZ}). Thus formula (\ref{Ht})  defines a semigroup of positive contractions on $L^p(\gamma_\infty)$ for every $p\ge 1$, which we shall also denote  by $(\cH^{Q,B}_t)_{t\ge 0}$.\par 
  In this paper we are concerned with the boundedness of the maximal operator
 $$
\cH^{Q,B}_*f(x)=\sup_{t\ge 0} \mod{\cH^{Q,B}_t f(x)}.
$$
It is well known that by Banach's principle (see \cite{G}) this maximal operator is a key tool to investigate the almost everywhere convergence of $\cH^{Q,B}_tf$ to $f$ as $t$ tends to $0$ for $f$ in $L^p(\gamma_\infty)$.\par
If the semigroup $(\cH^{Q,B}_t)_{t\ge 0}$ is symmetric, i.e. if $\cH^{Q,B}_t$ is self-adjoint on $L^2(\gamma_\infty)$ for every $t\ge 0$, then $\cH^{Q,B}_*$ is bounded on $L^p(\gamma_\infty)$ for every $p$ in $(1,\infty]$, by the  Littlewood-Paley-Stein theory for symmetric semigroups of contractions on all $L^p$ spaces \cite{S}.  Is the result still true if we drop the symmetry assumption?
 In the same  monograph \cite{S} Stein says that for general diffusion semigroups the condition of self-adjointness cannot be much modified. Indeed if one considers the semigroup of translations $\cT_t f(x)=f(x+t)$ on the one-dimensional torus $\BT$, for every $p$ in $[1,\infty]$ it is easy to construct a function $f$ in $L^p(\BT)$ such that $\sup_{t\geq 0}\mod{T_tf(x)}=\infty$ everywhere.
  Notice that $(\cT_t)_{t\geq0}$ is a semigroup of normal, actually unitary, operators. \par
  However, in Theorem \ref{mgr} below we show that Stein's proof of the maximal theorem for  semigroups of symmetric contractions on all $L^p(\mu)$, $1\le p\le \infty$, can be adapted to semigroups of {\it normal} contractions  such that the generator of the semigroup on $L^2(\mu)$ is a sectorial operator of angle $\phi<\pi/2$. 
Since the generator of the Ornstein-Uhlenbeck on $L^2(\gamma_\infty)$ is sectorial of angle strictly  less than $\pi/2$ this implies that if $(\cH^{Q,B}_t)_{t\ge0}$ is normal on $L^2(\gamma_\infty)$ then the maximal operator $\cH_*$ is bounded on $L^p(\gamma_\infty)$ for every $p$ in $(1,\infty]$.\par
It remains to investigate the boundedness of  the Ornstein-Uhlenbeck maximal operator $\cH^{Q,B}_*$ on $L^1(\gamma_\infty)$.  In Section \ref{s: strong} we show that $\cH^{Q,B}_*$ is always  unbounded on $L^1(\gamma_\infty)$. This still leaves open the question of the validity of the weak type $1$ estimate
$$
\gamma_\infty\big(\set{x\in\BR^d: \mod{\cH^{Q,B}_* f(x)}>\alpha}\big)\le\frac{C\,\norm{f}{1}}{\alpha}\qquad\forall f\in L^1(\gamma_\infty)\quad\forall \alpha>0.
$$\par
Even in the symmetric case very little is known about the weak type $1$ boundedness of the Ornstein-Uhlenbeck maximal operator. The only result which is known is for the semigroup with covariance matrix $Q=I$ and drift matrix $B=-I$ for which  the weak type $1$ boundedness of $\cH^{Q,B}_*$ is due to B.~Muckenhoupt  \cite{M}  in dimension one and to  P.~Sj\"ogren  \cite{Sj}   in arbitrary dimension. Sj\"ogren's proof was subsequently simplified in \cite{MPS} and \cite{GMMST}.  The arguments in these papers easily extend to the case where $B=-\lambda I$ for some $\lambda>0$. However, already the case where $B$ is a diagonal matrix with at least two different eigenvalues seems to require new ideas. \par

In this paper we investigate the weak type $1$ estimate for the maximal operator associated to  the Ornstein-Uhlenbeck semigroup with covariance matrix $Q=I$ and drift $B=-\lambda (I-R)$, where $\lambda>0$ and $R$ is a nonzero real $d\times d$ skew-adjoint matrix. 
The interest of these semigroups is motivated by  the fact that they are the basic building blocks of normal Ornstein-Uhlenbeck semigroups. Indeed, in Section \ref{s:prelim} we show that, after a change of variables, any normal Ornstein-Uhlenbeck semigroup can be written as the product of commuting semigroups of this form.\par
For these \emph{particular} semigroups we shall prove two results. First we shall prove that  the ``truncated" maximal operator
$$
\cH^{Q,B}_{*,{[0,T]}}f(x)=\sup_{t\in[0,T]}\mod{\cH^{Q,B}_t\,f(x)} 
$$
 is of weak type $1$.
Second, we shall prove that if the one-parameter group of rotations $(\e^{tR})_
{t\in\BR}$ generated by $R$ is periodic then the full maximal operator $\cH^{Q,B}_*$ is of weak type $1$.\par
Finally we mention that, by using the results of the present paper, in \cite{MN2} we have proved  that first order Riesz transforms associated to the generator of these \lq periodic\rq \, semigroups are of weak type $1$.\par
We now briefly describe the content of the paper. In Section \ref{s:prelim} we characterize the generators of normal Ornstein-Uhlenbeck semigroups and we show that, after a change of coordinates, normal semigroups are the product of commuting semigroups with covariance matrix $Q=I$ and drift $B=-\lambda(I-R)$, with $\lambda>0$ and $R$ a real skew-adjoint matrix. \par
In Section \ref{s: kernel} we give an explicit representation of the integral kernel of these semigroups with respect to the invariant measure. We show that, modulo an orthogonal change of coordinates,   the semigroup kernel is the product of the kernel of a symmetric semigroup and some two-dimensional kernels. Ultimately, this  will enable us to reduce the problem of the weak type $1$ boundedness of the maximal operator to proving estimates of kernels defined on $\BR^2\times\BR^2$.\par
In Section \ref{s: strong} we study the boundedness of the maximal operator $\cH^{Q,B}_*$ on $L^p(\gamma_\infty)$, $1\le p\le \infty$, for \emph{arbitrary} $Q$ and $B$. We prove that the truncated maximal operator is always unbounded on $L^1(\gamma_\infty)$ and that, when the semigroup is normal, the full maximal operator is bounded on $L^p(\gamma_\infty)$, $1<p\le \infty$.\par
Finally, in Section \ref{s:weaktype} we prove the weak type estimate for the truncated and the full maximal operator when $Q=I$ and $B=-\lambda(I-R)$. By the results of Section \ref{s: kernel} the kernel of the semigroup is a perturbation of the kernel of a symmetric semigroup. When $t$ is close to zero  the perturbation is small and the kernel of the nonsymmetric semigroup can be controlled by the kernel of the symmetric semigroup. The same thing happens in the periodic case when $t$ is close to an integer multiple of a period. This enables us to apply  the results of \cite{GMMST} to prove the weak type estimate for the truncated maximal operator and of the full maximal operator in the periodic case.

\section{Preliminaries} \label{s:prelim}
The Schwartz space $\cS(\BR^d)$ is a core for the infinitesimal generator $\cL_{Q,B}$ of the semigroup $(\cH^{Q,B}_t)_{t\ge0}$ on $L^p(\gamma_\infty)$ for every $p$, $1<p<\infty$, and
$$
\cL_{Q,B} f= \frac{1}{2} \tr(Q\nabla^2)f +\langle Bx,\nabla\rangle f \qquad\forall f\in\cS(\BR^d).
$$\par
 By a result of G.~Metafune, J.~Pr\"uss, A.~Rhandi and R.~Schnaubelt (see \cite[Lemma 2.2]{MPRS}) there exists a linear change of coordinates in $\BR^d$ which allows us to reduce the analysis of the operator $\cL_{Q,B}$ to the case where $Q=I$ and $Q_\infty$ is a diagonal matrix. Indeed, let $M_1$ be an invertible real matrix such that $M_1QM^*_1=I$ and $M_2$ an orthogonal matrix such that $M_2M_1Q_\infty M^*_1M_2=\diag(\lambda_1,\ldots,\lambda_d):=D_\lambda$ for some $\lambda_j>0$. Then, if we take $M=M_2M_1$ and we denote by $\Phi_M:\cS(\BR^d)\to \cS(\BR^d)$ the similarity transformation defined by $\Phi_M f(x)=f(M^{-1}x)$ we have that $\cL_{Q,B}=\Phi_M^{-1} \cL_{I,\tilde B}\,\Phi_M$ where
\begin{equation}\label{LtiBti}
\tilde B=-\frac{1}{2}D_{1/\lambda}+R
\end{equation}
and $R$ is a matrix such that 
\begin{equation}\label{skew}
RD_\lambda=-D_\lambda R^*.
\end{equation}
 The invariant measure for the semigroup generated by $\cL_{I,\tilde B}$ is
$$
\wrt\tilde\gamma_\infty(x)=(2\pi)^{-d/2}(\det D_\lambda)^{-1/2} \e^{-\frac{1}{2}\langle D_{\lambda}^{-1}x,x\rangle}\wrt\lambda(x).
$$
Moreover $\tilde\gamma_\infty(E)=\gamma_\infty(M^{-1}E))$ for every Borel subset $E$ of $\BR^d$ and $\Phi_M$ extends to an isometry of $L^{p}(\gamma_\infty)$ onto $L^{p}(\tilde\gamma_\infty)$.\par
By (\ref{LtiBti}) we may write the operator $\cL_{I,\tilde B}$ as the sum
\begin{equation}\label{s+a}
\cL_{I,\tilde B}=\cL^0+\cR,
\end{equation}
where $\cL^0=\frac{1}{2}\Delta-\frac{1}{2}\langle D_{1/\lambda}x,\nabla\rangle$ and $\cR=\langle Rx,\nabla\rangle$ are the symmetric and the antisymmetric part of $\cL_{I,\tilde B}$ on $L^2(\tilde\gamma_\infty)$, respectively. Thus, the operator $\cL_{Q,B}$ is symmetric on $L^2(\gamma_\infty)$ if and only if $R=0$. \par
Let $(\cH^{I,\tilde B}_t)_{t\ge 0}$ be the semigroup generated by $\cL_{I,\tilde B}$ and $\cH^{I,\tilde B}_*$ the corresponding maximal operator. Clearly, $\cH^{Q,B}_*$ is bounded on $L^p(\gamma_\infty)$ or of weak type $1$ with respect to $\gamma_\infty$ if and only if $\cH^{I,\tilde B}_*$ is bounded on $L^p(\tilde\gamma_\infty)$ or of weak type $1$ with respect to $\tilde\gamma_\infty$. 
Thus, the analysis of the maximal operator $\cH^{Q,B}_*$ may be reduced to the case where $Q=I$ and $\tilde Q_\infty=\diag\set{\lambda_1,\ldots,\lambda_d}$ for some $\lambda_j>0$. 
\begin{proposition}\label{normal}  Let $\tilde B$, $D_\lambda$ and $R$ be the matrices associated to $Q$ and $B$ as in (\ref{LtiBti}). Denote by $\cL^0$ and $\cR$ the symmetric and the antisymmetric part of $\cL_{I,\tilde B}$ as in (\ref{s+a}). Then
the following properties are equivalent
\begin{itemize}
\item[\rmi] the semigroup $(\cH^{Q,B}_t)_{t\ge0}$ is normal on $L^2(\gamma_\infty)$;
\item[\rmii] the symmetric and the antisymmetric parts of $\cL_{I,\tilde B}$ commute; i.e.
$$
[\cL^0,\cR]\phi=0\qquad\forall \phi\in \cS(\BR^d);
$$
\item[\rmiii] $R+R^*=0$;
\item[\rmiv]  $D_\lambda$ and $R$ commute.
\end{itemize}

\end{proposition}
\begin{proof}
  We claim that $\cL_{I,\tilde B}^*=\cL^0-\cR$. Indeed, on the one hand $(\cL^0)^*=\cL^0$ because $\cL^0$ is symmetric. On the other hand, integrating by parts, we get that
\begin{align*}
\cR^*&=-\cR+\langle Rx, D_\lambda^{-1}x\rangle-\tr R \\ 
&=-\cR,
\end{align*}
because $\tr R=0$ and $\langle Rx, D_\lambda^{-1}x\rangle=0$ since $\langle Rx, D_\lambda^{-1}x\rangle=\langle x, R^*D_\lambda^{-1}x\rangle=-\langle x, D_\lambda^{-1}Rx\rangle=-\langle D_\lambda^{-1} x, Rx\rangle$  by (\ref{skew}). \par
The semigroup $(\cH^{Q,B}_t)_{t\ge 0}$ is normal if and only if its generator $\cL_{Q,B}$ on $L^2(\gamma_\infty)$ is normal and this happens if and only if $\cL_{I,\tilde B}$ is normal on $L^2(\tilde\gamma_\infty)$, i.e. 
$[\cL_{I,\tilde B},\cL_{I,\tilde B}^*]\phi=0$ for all $\phi$ in $\cS(\BR^d)$. Now
\begin{align*}
[\cL_{I,\tilde B},\cL_{I,\tilde B}^*]=&[\cL^0+\cR,\cL^0-\cR] \\
=&2[\cR,\cL^0]. 
\end{align*}
This shows that \rmi\  and \rmii\  are equivalent. Next observe that
\begin{align}\label{comm}
[\cR,\cL^0]  
=&-\langle \nabla,R\nabla\rangle+\frac{1}{2}\langle(RD_{1/\lambda}-D_{1/\lambda}R)x,\nabla\rangle.
\end{align}
Hence $[\cR,\cL^0]  
$ vanishes  if and only if  $\langle \nabla,R\nabla\rangle$ and $\langle(RD_{1/\lambda}-D_{1/\lambda}R)x,\nabla\rangle$ both vanish, as can be easily seen by fixing any pair of indices $j$, $k$ and an arbitrary point $x_0$ and applying the commutator to a test function $\phi$ which in a neighbourhood of $x_0$ coincides with $(x-x_0)_j(x-x_0)_k$. Now,  $\langle \nabla, R\nabla\rangle $ vanishes if and only if $R+R^*=0$. Thus \rmii\  implies \rmiii. To prove the converse observe that by (\ref{skew}) the identity $R+R^*=0$ implies that $R$ and $D_{\lambda}$ commute. Thus also   $D_{1/\lambda}$ and $R$ commute. Hence $[\cR,\cL^0]=0$ by (\ref{comm}). Finally, if (\rm iv)\ holds then $R+R^*=0$ by (\ref{skew}). This concludes the proof of the proposition.
\end{proof}
In the last part of this section we show that operators of the form $\cL_{Q,B}$ with $Q=I$ and $B=\frac{1}{2\alpha}(R-I)$ where $\alpha>0$ and $R$ is a $d\times d$ skew-symmetric real matrix, are the basic building blocks of normal Ornstein-Uhlenbeck operators. This motivates the interest in studying the maximal operator associated to semigroups generated by them.\par
To simplify notation we write
\begin{equation}\label{L(a,R)}
\cL(\alpha,R)=\cL_{I,\frac{1}{2\alpha}(R-I)}=\frac{1}{2}\Delta-\frac{1}{2\alpha}\langle x,\nabla\rangle+\frac{1}{2\alpha}\langle Rx,\nabla\rangle,
\end{equation}
Let $(\cH^{Q,B}_t)_{t\ge 0}$ be a normal Ornstein-Uhlenbeck semigroup. By (\ref{LtiBti}) after a change of variables we may assume that its generator is of the form
$$
\cL_{I,\tilde B}=\frac{1}{2}\Delta-\frac{1}{2}\langle D_{1/\lambda}x,\nabla\rangle +\langle Rx,\nabla\rangle,
$$ 
where $R+R^*=0$ and  $R$ commutes with $D_{\lambda}$ by Proposition (\ref{normal}). Let $\alpha_1,\ldots,\alpha_\ell$ be the distinct eigenvalues of $D_\lambda$ and let 
$$
D_\lambda=\alpha_1 P_1+\cdots \alpha_\ell P_\ell
$$
be the spectral resolution of $D_\lambda$. The matrix $R$ commutes with the projections $P_j$ and if we set  $R_j={2\alpha}RP_j$ then  $R_j^*=-R^j$ and 
$
R=\frac{1}{2\alpha}\sum_{j=1}^\ell R_j.
$
Thus, denoting by $\Delta_j=\tr(P_j\nabla^2)$ and $\nabla_j=P_j\nabla$ the Laplacian and the gradient with respect to the variables in $P_j \BR^d$, we have
\begin{align*}
\cL_{I,\tilde B}=\sum_{j=1}^\ell \cL(\alpha_j, R_j),
\end{align*}
where
$$
 \cL(\alpha_j, R_j)=\frac{1}{2} \Delta_j-\frac{1}{2\alpha_j} \langle x, \nabla_j\rangle+\frac{1}{2\alpha_j}\langle R_jx,\nabla_j\rangle.
$$
The semigroup generated by $\cL_{I,\tilde B}$ is the product of the commuting semigroups  $(\e^{t\cL(\alpha_j,R_j)})_{t\ge 0}$ generated by  the operators $\cL(\alpha_j,R_j)$, $j=1,\ldots,\ell$, which are therefore the basic building blocks of normal Ornstein-Uhlenbeck semigroups.


 \section{The kernel of the semigroup with respect to the invariant measure} \label{s: kernel}
For our purposes it is convenient to write the Ornstein-Ulenbeck semigroup as a semigroup of integral operators with respect to the invariant measure $\gamma_\infty$.
 We recall that the Gauss measure with mean zero and covariance matrix $Q_t$ on $\BR^d$ is the measure 
 $$
\wrt\gamma_t(x)=(2\pi)^{-d/2} (\det Q_t)^{-1/2}\,\e^{-\frac{1}{2}\langle Q_t^{-1} x,x\rangle}\wrt\lambda(x) \qquad \forall t\in(0,\infty],
$$
where $\lambda$ denotes the Lebesgue measure. In the following, with a slight abuse of notation, we shall denote by the same symbol $\gamma_t$ also the density of the measure with respect to $\lambda$.
 A simple change of variables in (\ref{Ht}) yields
$$
\cH^{Q,B}_t f(x)=\int h_t(x,y) \,f(y)\wrt\gamma_\infty(y),
$$
where
\begin{equation}\label{kernel}
h_t(x,y)=\det(Q_\infty\,Q_t^{-1})^{1/2}\ \e^{ -\frac{1}{2}\left[\langle Q_t^{-1}(\e^{tB}x-y),(\e^{tB}x-y)\rangle-\langle Q_\infty^{-1}y,y\rangle\right]}
\end{equation}\par
The main result of this section is  that, after an orthogonal change of coordinates, the kernel of the semigroup generated by an operator of the form 
$$
\cL(\alpha,R)=\frac{1}{2}\Delta-\frac{1}{2\alpha}\langle x, \nabla\rangle
+\frac{1}{2\alpha}\langle Rx, \nabla\rangle,
$$
with $\alpha>0$  and $R+R^*=0$, can be written as the product of the kernel of the semigroup generated by its symmetric part $\cL(\alpha,0)$  and some two-dimensional kernels (see Theorem \ref{reduction} and formula (\ref{hprod})).
To simplify notation, for the rest of this section we write $\cL=\cL(\alpha,R)$ and $\cL^0=\cL(\alpha,0)$. Thus
$$
\cL^0=\frac{1}{2}\Delta-\frac{1}{2\alpha}\langle x, \nabla\rangle,\qquad \cL=\cL^0+\frac{1}{2\alpha}\langle Rx, \nabla\rangle
$$
Henceforth we shall denote by $(\e^{t\cL^0})_{t\ge 0}$ and by $(\e^{t\cL})_{t\ge 0}$ the semigroups generated by $\cL^0$ and by $\cL$, respectively, and by $h^0_t(x,y)$ and $h_t(x,y)$ their kernels with respect to the invariant measure
$$
\wrt\gamma_\infty(x)=(2\pi\alpha)^{-d/2} \ \e^{-\frac{\mod{x}^2}{2\alpha}}.
$$
By the results of the previous section, the operator $\cL^0$ is symmetric and $\cL$ is normal. \par 
To avoid having  many $\alpha$'s floating around and to be consistent with the notation in \cite{GMMST}, we fix $\alpha=1/2$. The formulas for arbitrary $\alpha>0$  can be obtained from this special case by replacing $t$ by $t/2\alpha$ and $(x,y)$ by $(x/\sqrt{2\alpha},y/\sqrt{2\alpha})$ in  formulas (\ref{kernelsimm}) and (\ref{kernelnor}) below. \par
The kernel of the semigroup $(\e^{t\cL^0})_{t\ge 0}$ is 
\begin{equation}\label{kernelsimm}
h_t^{0}(x,y)= (1-\e^{-2t})^{-d/2}\exp\left\{\frac{1}{2}\left[\frac{|x+y|^{2}}{\e^{t}+1}-\frac{|x-y|^{2}}{\e^{t}-1}\right]\right\}.
\end{equation}
The operator $\cR=\langle Rx, \nabla\rangle$ generates the semigroup of isometries $\e^{t\cR}f(x)=f(\e^{tR}x)$ of $L^p(\gamma_\infty)$, $1\le p\le \infty$. Since $\e^{t\cR}$  commutes with $\e^{t\cL_0}$ for every $t\ge 0$, the kernel of $(\e^{t\cL})_{t\ge 0}$ is 
\begin{equation}\label{kernelnor}
h_t(x,y)= h_t^{0}(\e^{tR}x,y).
\end{equation}
\par
We shall exploit the facts that the matrix $R$ is skew-adjoint and that  the symmetric semigroup $\sgr{\cL}$ commutes with orthogonal transformations to prove that, after an orthogonal change of coordinates, the operator $\cL$ and the kernel $h_t(x,y)$ can be written in a more convenient form. \par
First we consider a special  two-dimensional case. 
For every real number $\theta$ we denote by ${ R}(\theta)$ the $2\times2$ matrix
\begin{equation}\label{Rtheta}
{R}(\theta)=\left(\begin{array}{cc}\phantom{-}\!0 & \theta  \\-\theta & 0\end{array}\right).
\end{equation}

Let $x\land y$ denote the skew-symmetric bilinear form on $\BR^2$ defined by
$$
x\land y=x_1y_2-x_2y_1.
$$
Then
\begin{equation}\label{boh?}
\mod{\e^{tR(\theta)}x\pm y}^2=\mod{x}^2+\mod{y}^2+2\cos(t\theta)\langle x,y \rangle\pm \sin(t\theta) x\land y \qquad\forall x,y\in\BR^2.
\end{equation}
Now, consider the Ornstein-Uhlenbeck operator $\cL\big(\frac{1}{2},R(\theta)\big)$ on $\BR^2$. To simplify notation henceforth we write $\cL_\theta=\cL\big(\frac{1}{2},R(\theta)\big)$. Thus 
$$
\cL_\theta=\frac{1}{2}\Delta-\langle x,\nabla\rangle+\langle R(\theta)x,\nabla\rangle,
$$
is the operator with covariance matrix $Q=I$ and drift $B=-I+R(\theta)$.
By using  (\ref{kernelsimm}), (\ref{kernelnor}) and (\ref{boh?}) it is straigthforward to see that the kernel of the semigroup generated by $\cL_\theta$ is
\begin{equation}\label{factor}
h_{t}^\theta(x,y)=h_{t}^{0}(x,y)\  k_{t\theta}(x,y),
\end{equation}
where $h^0_t(x,y)$
is as in (\ref{kernelsimm}) with $d=2$  and
\begin{equation}\label{K_t}
k_{t\theta}(x,y)=\exp\left\{-\frac{\e^{-t}}{1-\e^{-2t}}\big[\big(1-\cos 
(t\theta)\big)\langle x, y\rangle+\sin (t\theta)\,x\land y\big]\right\}\,.
\end{equation}
\par
Next we consider the case when the matrix $R$ is a $d\times d$ matrix in block diagonal form, with $2\times 2$ blocks of  the form (\ref{Rtheta}).
Let $n=\pint{d/2}$ be the greatest integer less than or equal to $d/2$. If  $\Theta=(\theta_1,\ldots,\theta_n)$ is in $\BR^n$ we denote by $R(\Theta)$ the $d\times d$ block-diagonal matrix 
$$
\left(\begin{array}{ccccc}
R(\theta_{1}) & \, & \, & \, & \,\\
\, & \cdot & \, & \, & \,\\
\, & \, & \cdot & \, & \, \\
\, & \, & \, & \cdot & \, \\
\, & \, & \, & \,& R(\theta_{n})
\end{array}\right)\,\textrm{ or }\,
\left(\begin{array}{cccccc}
R(\theta_{1}) & \, & \, & \, & \,\\
\, & \cdot & \, & \, & \, & \,\\
\, & \, & \cdot & \, & \, & \,\\
\, & \, & \, & \cdot & \, & \,\\
\, & \, & \, & \,& R(\theta_{n}) & \, \\
\, & \, & \, & \, & \, & \, 0
\end{array}\right)
$$
according to  whether $d$ is even or  odd, respectively.  \par
Assume first that $d$ is even. Given a vector $x$ in $\BR^d\simeq(\BR^{2})^n$ we write $x=(\vett{\xi}{n})$, where $\xi_k=(x_{2k-1},x_{2k})\in\BR^2$ for $k=1,\ldots,n$. Let $\cL_\Theta=\cL\big(\frac{1}{2},R(\Theta)\big)$ be the Ornstein-Uhlenbeck operator on $\BR^d$ of the form
\begin{equation}\label{LT}
\cL_\Theta=\frac{1}{2}\Delta-\langle x,\nabla\rangle+\langle R(\Theta)x,\nabla\rangle.
\end{equation}
{\sloppy
Then  $\cL_\Theta=\cL_{\theta_1}+\ldots+\cL_{\theta_n}$   where each $\cL_{\theta_k}$ for $k=1,\ldots, n$ is a two-dimensional Ornstein-Uhlenbeck operator acting in the variables  $\xi_k=(x_{2k-1},x_{2k})$ of the form
$$
\cL_{\theta_k}=\frac{1}{2}\Delta_k-\langle \xi_k,\nabla_k\rangle+\langle R(\theta_k)\xi_k,\nabla_k\rangle.
$$
}
Here $\Delta_k$ and $\nabla_k$ denote the two-dimensional Laplacian and gradient in the variables $(x_{2k-1},x_{2k})$.\par
Thus the operators  $\cL_{\theta_k}$,  $k=1,\ldots, n$ commute as do the semigroups 
generated by them. This implies that the kernel $h^{\Theta}_t(x,y)$ of the semigroup $(\e^{t\cL_{\Theta}})_{t\ge0}$ is 
the product of the  kernels of the semigroups $(\e^{t\cL_{\theta_k}})_{t\ge 0}$,  $k=1,\ldots, n$; i.e.
$$
h_t^{\Theta}(x,y)= \prod_{k=1}^n h_t^{\theta_k}(\xi_k,\eta_k)
$$
with $\xi_k=(x_{2k-1},x_{2k})$ and $\eta_k=(y_{2k-1},y_{2k})$ in $\BR^2$, where $h_t^{\theta_k}(\xi_k,\eta_k)$ are as in (\ref{factor}).
\par
If $d$ is odd then $\cL_\Theta=\cL_{\theta_1}+\ldots+\cL_{\theta_n}+\cL_{n+1}$  where  $\cL_{\theta_k}$, $k=1,\ldots,n$, are as before and $\cL_{n+1}$ is  the one-dimensional symmetric  Ornstein-Uhlenbeck operator $\frac{1}{2}\partial_{x_{n+1}}^2-x_{n+1}\partial_{x_{n+1}}$ acting in the variable $x_{n+1}$. Thus the kernel $h_t(x,y)$ has an additional factor 
$
h^{0}_t(x_{n+1},y_{n+1})
$, which is the kernel of a one-dimensional symmetric Ornstein-Uhlenbeck semigroup.\par
In any case, regardless of the parity of $d$, by (\ref{factor}) we may write the kernel of $\e^{t\cL_\Theta}$ in the following way
\begin{align}
h^\Theta_t(x,y)&=h^{0}_t(x,y)\prod_{j=1}^n k_{t\theta_j}(\xi_j,\eta_j)\nonumber\\
&=h^{0}_t(x,y)\prod_{\theta_j\not=0} k_{t\theta_j}(\xi_j,\eta_j)\label{hprod},
\end{align}
where $h^{0}_t(x,y)$ is the kernel of the $d$-dimensional symmetric semigroup generated by $\frac{1}{2}\Delta-\langle x,\nabla\rangle$ and each $k_{t\theta_j}$ is a two-dimensional kernel as in  (\ref{K_t}). 
\par
Finally, we show that  the analysis of any  operator $\cL=\frac{1}{2}\Delta-\langle x,\nabla\rangle +\langle Rx,\nabla\rangle$, where $R$ is a skew adjoint matrix, may be reduced to that of an  operator of the form $\cL_\Theta$. As in  Section \ref{s:prelim}, given an invertible real $d\times d$-matrix $M$, we denote by $\Phi_M:C(\BR^{d})\to C(\BR^{d})$ the transformation defined by $\Phi_{M}u(y)=u(M^{-1}y)$.

\begin{theorem}\label{reduction} Let $n=\pint{d/2}$ be the greatest integer less than or equal to $d/2$ and let $\cL$ be the operator $\frac{1}{2}\Delta- \langle x,\nabla\rangle +\langle Rx,\nabla\rangle$, where 
 $R$ is a $d\times d$ real, skew-adjoint matrix. Then there exists a $d\times d$ orthogonal matrix $g$ and a vector $\Theta=(\vett{\theta}{n})$ with $\theta_j\geq 0$ such that $\Phi_g \cL\Phi_g^{-1}=\cL_\Theta$. Moreover the kernels $h_t(x,y)$ and $h_t^\Theta(x,y)$ of the semigroups generated by $\cL$ and $\cL_\Theta$, respectively, satisfy the identity
 $$
h_t(x,y)=h_t^\Theta(gx,gy) \qquad\forall x,y\in \BR^d, \ t>0.
$$
\end{theorem}
\begin{proof} 
The set $\mathfrak{a}=\set{R(\Theta): \Theta\in\BR^n}$ is a maximal abelian subalgebra of the Lie algebra $\mathfrak{so}(d)$ of skew-symmetric $d\times d$ matrices. Since, by a well known result of Lie algebras (see \cite{C}), every element of $\mathfrak{so}(d)$ is conjugated to an element of $\mathfrak{a}^+=\set{R(\Theta): \Theta\in \overline{\BR_+}\,^n}$, given a skew-symmetric matrix $R$ there exists an orthogonal matrix $g$ and a vector $\Theta=(\theta_1,\ldots,\theta_n)$, with $\theta_j\geq 0$, such that $R=gR(\Theta)g^{-1}$. The identity $\Phi_g \cL\Phi_g^{-1}=\cL_\Theta$ follows, because the symmetric part $\frac{1}{2}\Delta-\langle x,\nabla\rangle$ of the operator $\cL$ commutes with $\Phi_g$.\par
This implies that $\Phi_g \e^{t\cL}\Phi_g^{-1}=\e^{t\cL_\Theta}$ for every $t\geq 0$. The identity between the kernels of the semigroups follows immediately from it.\par
\end{proof}

\section{Strong type estimates} \label{s: strong}
In this section we return to consider a Ornstein-Uhlenbeck semigroup $(\cH^{Q,B}_t)_{t\ge0}$ with arbitrary covariance $Q$ and drift $B$. We prove that the truncated Ornstein-Uhlenbeck maximal operator $\cH^{Q,B}_{*,{[0,T]}}$ is always unbounded on $L^1(\gamma_\infty)$ and when the semigroup is normal  the full maximal operator $\cH^{Q,B}_*$ is bounded on $L^p(\gamma_\infty)$, $1<p\le\infty$. 
\begin{theorem}\label{unbL1}
For all $T>0$ the operator $\cH^{Q,B}_{*,{[0,T]}}$ is unbounded on $L^1(\gamma_\infty)$.
\end{theorem}
\begin{proof} Suppose, by contradiction, that $\cH^{Q,B}_{*,{[0,T]}}$ is bounded on $L^1(\gamma_\infty)$ for some $T>0$. Denote by $\gamma_\infty$ the density of the invariant measure with respect to the Lebesgue measure.
Let $(f_n)$ be a sequence of nonnegative functions of norm $1$ in $L^1(\gamma_\infty)$ which converges in sense of distributions to $\gamma_\infty(0)^{-1}\delta_0$. Then  there exists a constant $C$ such that $\norm{\cH^{Q,B}_{*,{[0,T]}} f_n}{1}\le C$ for every $n$. Moreover
\begin{align*}
\lim_{n\to\infty}\cH^{Q,B}_tf_n(x)=&\lim_{n\to\infty} \int h_t(x,y)\,f_n(y)\wrt\gamma_\infty(y)
=h_t(x,0)
\end{align*}
uniformly on compact subsets of $\BR^d$. Thus, for $n$ sufficiently large,
$$
 \cH^{Q,B}_{*,{[0,T]}} f_n(x)\ge \cH^{Q,B}_tf_n(x)\ge h_t(x,0)-1 \qquad\forall x\in B(0,1)\quad \forall  t\in[0,T]. 
$$
Hence 
\begin{equation}\label{inth_t}
\int_{\mod{x}\le 1} \sup_{t\in[0,T]}h_t(x,0)\wrt\gamma_\infty(x)\le C.
\end{equation}
Now recall  the expression of the kernel $h_t(x,y)$  given in (\ref{kernel}).
Since $Q_t\sim t\,Q$ for $t\to0^+$, if $t\in(0,\epsilon)$ for some $\epsilon>0$ sufficiently small then there exist positive constants $c_0, c_1$ and $c_2$ such that
\begin{align*}
h_t(x,0)&=\left( \frac{\det Q_\infty}{\det Q_t}\right)^{1/2} \exp\left\{-\frac{1}{4}\langle Q_t^{-1} \e^{tB}x,\e^{tB}x\rangle\right\} \\ 
&\ge c_0\  t^{-d/2} \exp\left\{-c_1\frac{\mod{\e^{tB}x}^2}{t}\right\}  \\ 
&\ge c_0 \  t^{-d/2} \exp\left\{-c_2\frac{\mod{x}^2}{t}\right\}. \end{align*}
Thus if $\mod{x}\le1$
$$
\sup_{0<t<\epsilon} h_t(x,0)\ge\ c_0\ \sup_{0<t<\epsilon} t^{-d/2} \e^{-c_2\frac{\mod{x}^2}{t}}\ge c_\epsilon \mod{x}^{-d},
$$
which contradicts (\ref{inth_t}).
\end{proof}
The positive result for $L^p(\gamma_\infty)$, $1<p\le \infty$, for normal Ornstein-Uhlenbeck semigroups follows from a more general result for normal semigroups  of contractions  
on all $L^p$-spaces, whose generator on $L^2$ is sectorial. Indeed we  
have the following theorem.
\begin{theorem}\label{mgr}
Let $(X,\mu)$ be a $\sigma$-finite measure space.
Let $(T_t)_{t\geq 0}$ be a semigroup of contractions on $L^p(\mu)$ for every  
$p$ in $[1,\infty]$, which is strongly continuous for $p<\infty$.  
Suppose that each $T_t$ is normal on $L^2(\mu)$ and that  the spectrum  
of  the generator $\cG$ on $L^2(\mu)$ is contained in the sector  
$-\overline{S}_\theta$ for some $\theta\in[0,\pi/2)$. Then the maximal  
operator
$$
T_*f(x)=\sup_{t>0}\mod{T_tf(x)}.
$$
  is bounded on $L^p(\mu)$ for $1<p\leq\infty$.
\end{theorem}
\begin{proof}
By examining carefully Stein's proof of the maximal theorem for  
self-adjoint semigroups of contractions (see \cite[p. 73--81]{S})  
one realizes that self-adjointness plays a r\^ole only in the proof of  
the boundedness on $L^2(\mu)$ of the Littlewood-Paley functions
$$
g_k(f)(x)=\left(\int_0^\infty\mod{t^k D_t^k T_t f(x)}^2\frac{\wrt  
t}{t}\right)^{1/2}, \qquad\forall k=1,2,\ldots
$$
However, the same result can  also be obtained under the assumptions of  
the theorem. Indeed, let
$$
-\cG=\int_{\overline{S}_\theta} z\wrt\cP_z
$$
be the spectral resolution of $-\cG$.   By the spectral theorem for  
normal operators
$$
D_t^kT_tf=(-1)^k\int_{\overline{S}^+_\theta} z^k \e^{-tz} \wrt \cP_z f,
$$
where $\overline{S}^+_\theta=\overline{S}_\theta\setminus\set{0}$. 
Hence
$$
\norm{D^k_t T_t f}{2}\!\!^2=\int_{\overline{S}^+_\theta}\mod{z}^{2k}\  
\e^{-2t\Re z} \langle\wrt\cP_z f,f\rangle.
$$
Thus
\begin{align*}
\int_X \mod{g_k(f)(x)}^2\wrt\mu(x)&=\int_X\int_0^\infty\mod{t^k  
D_t^kT_t f(x)}^2 \frac{\wrt t}{t}\wrt\mu(x) \\
&=\int_0^\infty t^{2k} \int_X \mod{D_t^k T_t f(x)}^2  
\wrt\mu(x)\frac{\wrt t}{t} \\
&= \int_0^\infty\int_{\overline{S}^+_\theta}\mod{tz}^{2k}\ \e^{-2t\Re  
z}\langle\wrt\cP_z f, f\rangle\frac{\wrt t}{t}\\
&=\int_{\overline{S}^+_\theta} \int_0^\infty\mod{tz}^{2k}\ \e^{-2t\Re  
z}\frac{\wrt t}{t}\langle\wrt\cP_z f, f\rangle\\
&\leq  
\frac{\Gamma(2k)}{(2\cos\theta)^{2k}}\int_{\overline{S}^+_\theta}\langle\wrt\cP_z f, f\rangle\\
&\leq \frac{\Gamma(2k)}{(2\cos\theta)^{2k}}\norm{f}{2}\!\!^2,
\end{align*}
because $\mod{z}\leq (\cos\theta)^{-1}{\Re z}$ in $\overline{S}_\theta$.  
This proves that $f\mapsto g_k(f)$ is bounded on $L^2(\mu)$. The rest  
of the proof is just as in \cite[p. 76--81]{S}.\par
\end{proof}
{\sloppy
\begin{corollary}\label{NOUonL^p}
Let $(\cH^{Q,B}_t)_{t\ge 0}$ be a normal Ornstein-Uhlenbeck semigroup. 
 Then the maximal operator $\cH^{Q,B}_*$ is bounded on $L^p(\gamma_\infty)$ for every $p$ in $(1,\infty)$. 
\end{corollary}
}
\begin{proof}
By \cite{MPP} the spectrum of the generator of $(\cH^{Q,B}_t)_{t\ge 0}$ is contained in a sector of angle less than $\pi/2$. Hence the conclusion follows from Theorem \ref{mgr}.
\end{proof}
\section{The weak type estimate.}\label{s:weaktype} 
  In this section we shall prove the weak type $1$ estimate for the maximal operators associated to the normal Ornstein-Uhlenbeck semigroup $(\cH^{Q,B}_t)_{t\ge0}$ with covariance $Q=I$ and drift $B=\frac{1}{2\alpha}(R-I)$, where $\alpha>0$ and $R$ is a skew-symmetric real matrix, i.e. for the semigroup generated by the operator 
$$
\cL(\alpha,R)=\frac{1}{2}\Delta-\frac{1}{2\alpha}\langle x,\nabla\rangle+\frac{1}{2\alpha}\langle Rx,\nabla\rangle.
$$
 Namely, we shall prove the following theorem.
\begin{theorem}\label{wt1} 
For every $T>0$ the truncated maximal operator $$\cH_{*,{[0,T]}} f(x)=\sup_{t\in[0,T]}|\e^{t\cL(\alpha,R)} f(x)|$$ is of weak type $1$. 
If the one-parameter group $(\e^{tR})_{t\in\BR}$ is periodic then the full maximal operator 
$
\cH_* f(x)=\sup_{t\ge0}|\e^{t\cL(\alpha,R)} f(x)|
$
 is of weak type $1$.
\end{theorem} 
As we have already remarked in Section \ref{s: kernel} we may assume that $2\alpha=1$, by a scaling argument.
\par
First we reduce the problem to proving that two smaller maximal operators are of weak type $1$. For every subset $A$ of $\BR_+$ denote by $\cH_{*,A}$ the maximal operator defined by
\begin{equation*}
\cH_{*,A}f(x)=\sup_{t\in A}|\e^{t\cL(1/2,R)} f(x)|,\quad f\in L^{1}(\ga).
\end{equation*}
If $I$ is a closed interval in $\BR_{+}$ and $P$ is a positive number, we denote by $I_P^{\sharp}$ the union of $ P\BN$-translates of $I$, i.e. $I^{\sharp}=\bigcup_{n\in\BN}(I+P n)$.
\begin{lemma}\label{RIDUXMAXOP}
Suppose that for some $t_{0}>0$   the maximal operator $\cH_{*,[0,t_{0}]}$  is of weak type $1$. Then the truncated maximal operator $\cH_{*,[0,T]}$ is of weak type $1$ for every $T>0$. If, furthermore, there exists  an interval $I$ in $\BR_{+}$ such that the operator $\cH_{*,I_P^\sharp}$ is of weak type $1$ then the full maximal operator  $\cH_*$ is of weak type~$1$.
\end{lemma}
\begin{proof}
First we show that if $A$ is a subset of $\BR_{+}$ such that the operator $\cH_{*,A}$ is of weak type $1$ and $B=\bigcup_{i=1}^{N}(A+t_{i})$ is a finite union of translates of $A$ then $\cH_{*,B}$ is of weak type $1$. Indeed
\begin{align*}
\cH_{*,B}f(x)=\sup_{t\in B}|\e^{t\cL(1/2,R)} f(x)|&=\max_{i=1,\ldots,N}\sup_{t\in A}|\e^{(t+t_{i})\cL(1/2,R)}f(x)|\\
&=\max_{i=1,\ldots,N}\sup_{t\in A}|\e^{t\cL(1/2,R)} \e^{t_{i}\cL(1/2,R)}f(x)|\\
&=\max_{i=1,\ldots,N}\cH_{*,A}\ \e^{t_{i}\cL(1/2,R)}f(x).
\end{align*}
Hence, for $\lambda>0$ fixed,
\begin{align*}
\ga(\{x\in\BR^{d}:\cH_{*,B}f(x)>\lambda\})&\leq\sum_{i=1}^{N}\ga(\{x\in\BR^{d}:\cH_{*,A}\ \e^{t_{i}\cL(1/2,R)}f(x)>\lambda\})\\
&\leq\frac{C}{\lambda}\sum_{i=1}^{N}\|\e^{t_{i}\cL(1/2,R)} f\|_{L^{1}(\ga)}\\
&\leq\frac{C\,N}{\lambda}\|f\|_{L^{1}(\ga)},
\end{align*}
because $\e^{t_{i }\cL(1/2,R)}$ is a contraction on $L^{1}(\ga)$ for every $i=1,\ldots,N$.\par
The conclusion follows  because the set $[0,T]$ is a finite union of translates of $(0,t_{0})$ and $\BR_+$ is a finite union of translates of $[0,T]$ and $I_P^{\sharp}$. 
\end{proof}
\par
Thus we only need to prove the weak type $1$ estimate for the operator $\cH_{*,A}$ when $A=(0,t_0)$ and $A={I_P^{\sharp}}$ for some $t_0>0$ and some closed interval $I$ in $\BR_+$. As in the analysis of the maximal operator for the symmetric Ornstein-Uhlenbeck semigroup $(\e^{t\cL(1/2,0)})_{t\geq0}$ (see \cite{GMMST}), we shall decompose each of these two maximal operators in a ``local'' part, given by a kernel living close to the diagonal, and the remaining or ``global'' part. To this end consider the set
$$
L=\set{(x,y)\in\BR^d\times\BR^d: \mod{x-y}\leq\min(1,\mod{x+y}^{-1})}
$$ 
and  denote by $G$ its complement. We shall call $L$  and $G$ the `local' and the `global' region, respectively.  The local
and the global parts of the operator $\cH_{*,A}$ are
defined by
\begin{align}
\cH_{*,A}^{\loc} f(x)=&\sup_{t\in A}\left\vert\int
h_t(x,y) \1_{L}(x,y) f(y)
\wrt \gamma(y)\right\vert    \nonumber
\\
\cH_{*,A}^{\glob}
f(x)=&\sup_{t\in A}\left\vert\int
h_t(x,y) \1_{G}(x,y) f(y)
\wrt \gamma(y)\right\vert ,  \label{HstarG}
\end{align}
where $\1_{L}$ and $\1_{G}$ are the characteristic functions of the sets
$L$ and $G$ respectively. Clearly
$$
\cH_{*,A} f(x)\le \cH_{*,A}^{\loc} f(x)+\cH_{*,A}^{\glob} f(x).
$$
We shall prove separately the weak type $1$ estimate for $\cH_{*,A}^{\loc}$ and $\cH_{*,A}^{\glob}$. \par
First we deal with the local part. 
We shall actually prove that for all Ornstein-Uhlenbeck semigroups $(\cH_t)_{t\ge0}$, without restrictions on covariance and drift, the local maximal operator $\cH^{\loc}_*=\cH_{*,\BR_+}^{\loc}$ is of weak type $1$.
\begin{lemma}\label{l:locest}
Let $(\cH_t)_{t\ge 0}$ be a Ornstein-Uhlenbeck semigroup with arbitrary covariance and drift. Then there exist positive constants $c$ and $C$ such that for all $(x,y)$ in the local region $L$
\begin{equation}\label{locest}
h_t(x,y)\le C\,(1-\e^{-t})^{-d/2} \ \gamma_\infty(y)^{-1}\ \exp\left(-c\frac{\mod{x-y}^2}{1-\e^{-t}}\right) \qquad\forall t>0.
\end{equation}
\end{lemma}
\begin{proof} 
Since the real part of the eigenvalues of $B$ is negative, there exist  positive constants $\alpha\le \beta$ and $C_0$ such that $C_0^{-1}\,\e^{2\alpha s}\mod{x}^2\le C_0\,\mod{\e^{sB^*}}\le \e^{2\beta s}\mod{x}^2$ for all $x\in\BR^d$ and all $s\in \BR$. Thus, by (\ref{Qt}) there exists a  positive constant $C$  such that 
$$
C^{-1}(1-\e^{- t}) I\le Q_t\le C(1-\e^{-t}) I\qquad\forall t\in(0,\infty].
$$ 
and, by (\ref{kernel}), there exist two positive constants $c$ and $C$ such that
\begin{equation}\label{}
h_t(x,y)\le C\,(1-\e^{-t})^{-d/2} \ \gamma_\infty(y)^{-1}\ \exp\left(-c\frac{\mod{e^{tB}x-y}^2}{1-\e^{-t}}\right).
\end{equation}
Now, for all $(x,y)$
 in the local region $L$
\begin{align*}
|e^{tB}x-y|^2&=|x-y+(e^{tB}-I)x|^{2}\\
&=|x-y|^{2}+|(e^{tB}-I)x|^{2}+2\langle x-y, (e^{tB}-I)x\rangle\\
&\geq|x-y|^{2}-2\|e^{tB}-I\||x-y\|x|\\
&\geq|x-y|^{2}-C(1-e^{-t}),
\end{align*}
because $\|e^{tB}-I\|\leq C(1-e^{-t})$ and $\mod{x-y}\,\mod{x}\leq C$ in the local region $L$.
\end{proof}
\begin{proposition}\label{weak1loc}
Let $(\cH_t)_{t\ge 0}$ be a Ornstein-Uhlenbeck semigroup with arbitrary covariance and drift. Then the maximal operator $\cH^{\loc}_*$ is of weak type $1$.
\end{proposition}
\begin{proof}
By Lemma \ref{l:locest} one has that for each $f\ge0$
\begin{align*}
\cH^{\loc}_* f(x)&\le\ C\ \sup_{0<s\le 1} s^{-d/2} \int\e^{-c\frac{\mod{x-y}^2}{s}}\, \1_L(x,y) \,f(y)\wrt\lambda(y) \\ 
&=\cW f(x),
\end{align*}
say. Since the operator $\cW$ is of weak type $1$ with respect to the Lebesgue measure and its kernel is supported in the local region $L$, the conclusion follows by well-known arguments (see for instance \cite[Section 3]{GMMST}).
\end{proof}
Now we turn to the proof of the weak type estimate for the global part of the maximal operator associated to the semigroup generated by the special Ornstein-Uhlenbeck operator 
$$
\cL(1/2,R)=\frac{1}{2}\Delta-\langle x,\nabla\rangle+\langle Rx,\nabla\rangle,
$$
where  $R$ is a skew-symmetric real matrix. \par
As in Section \ref{s: kernel} we denote by $h_t(x,y)$ and by $h^0_t(x,y)$ the kernels with respect to the invariant measure of the semigroups generated by $\cL(1/2,R)$  and by its symmetric part
$$
\cL^0=\frac{1}{2}\Delta-\langle x,\nabla\rangle,
$$
respectively (see \ref{kernelsimm} and \ref{kernelnor}).
\par
To estimate the semigroup kernel  in the global region, it is convenient to simplify the
expression of $h^{0}_t(x,y)$ by means of  the change of variables in the parameter $t$ introduced in \cite{GMMST}.
We denote by $\tau$ the function defined by
\begin{equation}\label{tau}
\tau(s)=
\log\frac{1+s}{1-s}\qquad s\in(0,1).
\end{equation}
Notice that  $\tau$ maps $(0,1)$ onto $\BR_+.$ It
is straightforward to check (see \cite{GMMST}) that for all $s$ in $(0,1)$
\begin{equation}\label{htau}
h^{0}_{\tau(s)}(x,y)=(4s)^{-d/2}(1+s)^d  \e^{\frac{|x|^2+|y|^2}{
2}-\frac{1}{4}\bigl(s|x+y|^2+\frac{1}{s}|x-y|^2\bigr)}.
\end{equation}
Next, as in \cite{GMMST}, we introduce the quadratic form
\begin{equation}\label{formaQ_s}
\cQ_{s}(x,y)=|(1+s)x-(1-s)y|^{2},\quad x,y\in\BR^{d}.
\end{equation}
Thus
$$
s|x+y|^{2}+\frac{1}{s}|x-y|^{2}=\frac{1}{s}\cQ_{s}(x,y)-2|x|^{2}+2|y|^{2}.
$$
and
\begin{equation}\label{symmQt}
h_{\tau(s)}^{0}(x,y)=s^{-d/2}\exp\left\{|x|^{2}-\frac{1}{4s}\cQ_{s}(x,y)\right\}\quad\forall s\in(0,1).
\end{equation}
\begin{lemma}\label{0t_0glob}
If $t_{0}>0$ is sufficiently small, there exists a positive constant $C$ such that for all $s$ in $(0,\tau^{-1}(t_{0}))$ and all $(x,y)$ in $\BR^{d}\times\BR^{d}$
\begin{equation}\label{GP,1}
h_{\tau(s)}(x,y)\leq Cs^{-\frac{d}{2}}\ \e^{|x|^{2}-\frac{1}{40\,s}\cQ_{s}(x,y)}.
\end{equation}
\end{lemma}
\begin{proof}
Let $n=\pint{d/2}$. The right hand side of the inequality to prove is invariant under orthogonal transformations. Hence, by Theorem \ref{reduction}, it is enough to prove the inequality for the kernel $h^\Theta_t(x,y)$, with $\Theta=(\vett{\theta}{n})\in\BR^n$, $\theta_j\geq 0$.  \par
By (\ref{hprod}) and (\ref{symmQt})
$$
h^\Theta_{\tau(s)}(x,y)\le \, s^{-d/2}\exp\left\{|x|^{2}-\frac{1}{4s}\cQ_{s}(x,y)\right\}\prod_{\theta_j>0} k_{\tau(s)\theta_j}(\xi_j,\eta_j),
  \quad\forall s\in(0,1),
$$
where $\xi_j=(x_{2j-1},x_{2j})$ and $\eta_k=(y_{2j-1},y_{2j})$ are in $\BR^2$ and each $k_{t\theta_j}$ is a two-dimensional kernel as in  (\ref{K_t}). \par
Define
\begin{equation}\label{E(x,y)}
M_s(x,y)=\exp\left\{-\frac{9}{40\,s}\cQ_{s}(x,y)\right\}\prod_{\theta_j>0}k_{\tau(s)\theta_j}(\xi_j,\eta_j).
\end{equation}
Then
$$
h^\Theta_{\tau(s)}(x,y)\le \, s^{-d/2}\exp\left\{|x|^{2}-\frac{1}{40\,s}\cQ_{s}(x,y)\right\}\ M_s(x,y),
$$
and to conclude the proof of the lemma all we need to show is that there exist a $s_0>0$ sufficiently small and a constant $C$ such that
\begin{equation}\label{E<1}
M_s(x,y)\le C \qquad\forall s\in(0,s_0) \ \forall (x,y)\in \BR^d\times \BR^d.
\end{equation}
Let us denote by $\cQ_s^{(m)}$ the quadratic form defined in (\ref{formaQ_s}) when considered as a function on $\BR^m\times \BR^m$. Then
$$
\cQ_s^{(d)}(x,y)=
\begin{cases}
\sum_{j=1}^n \cQ_s^{(2)}(\xi_j,\eta_j) & {\rm if }\ d\  {\rm is\  even}\\
 \\
\sum_{j=1}^n \cQ_s^{(2)}(\xi_j,\eta_j)+\cQ_s^{(1)}(x_{n+1},y_{n+1}) & {\rm if }\ d\  {\rm is\  odd.}\\
\end{cases}
$$
Thus, since $\cQ_s^{(m)}$ is nonnegative,
$$
M_s(x,y)\le \prod_{\theta_j>0} \exp\left\{-\frac{9}{40\,s}\cQ^{(2)}_{s}(\xi_j,\eta_j)\right\}k_{\tau(s)\theta_j}(\xi_j,\eta_j)
$$
regardless of the parity of $d$.
Hence we only need to show that each factor is bounded, i.e. that for every $\theta> 0$ there exist $s_0\in(0,1)$ and a constant $C$ such that for all  $(x,y)\in \BR^2\times \BR^2$
\begin{equation}\label{E2}
 \exp\left\{-\frac{9}{40\,s}\cQ_{s}(x,y)\right\}k_{\tau(s)\theta}(x,y)\le C\qquad\forall s\in(0,s_0),
\end{equation}
 where now $\cQ_s=\cQ_s^{(2)}$, for the sake of brevity.  \par 
To this end we fix $\beta$ in $(0,1)$, we let $\delta$ be a constant in $(0,1)$ to be chosen later and we denote by $\vartheta=\vartheta(x,y)$ the angle between the two vectors $x$ and $y$. The set $\BR^2\times \BR^2$ is the disjoint union of the five sets
\begin{align*}
R_{1}&=\{(x,y)\in \BR^{2}\times\BR^{2}:\,\langle x,y\rangle<0\}\\
R_{2}&=\{(x,y)\in \BR^{2}\times\BR^{2}:\,\langle x,y\rangle\geq0,\,x\land y\geq0\},\\
R_{3}&=\{(x,y)\in \BR^{2}\times\BR^{2}:\,\langle x,y\rangle\geq0,\,x\land y<0,\,|x-y|\geq\beta|y|\},\\
R_{4}&=\{(x,y)\in \BR^{2}\times\BR^{2}:\,\langle x,y\rangle\geq0,\,x\land y<0,\,|x-y|<\beta|y|,\,|\sin\vartheta|\geq\delta\},\\
R_{5}&=\{(x,y)\in \BR^{2}\times\BR^{2}:\,\langle x,y\rangle\geq0,\,x\land y<0,\,|x-y|<\beta|y|,\,|\sin\vartheta|<\delta\}.
\end{align*}
\par
We shall prove that (\ref{E2}) holds in each region $R_j$, $j=1,\ldots,5$.
Note  that by (\ref{K_t}) and (\ref{tau})
\begin{equation}\label{turna}
k_{\tau(s)\theta}(x,y)=\e^{-\frac{1-s^2}{4s}[(1-\cos(\tau(s)\theta))\langle x,y\rangle+\sin(\tau(s)\theta)\,x\land y]}
\end{equation}
and that the function $s\mapsto\tau(s)$ is positive and increasing in $(0,1)$ and $\tau(s)\sim 2s$ as $s\to 0^+$.
To prove the estimate in $R_{1}$, we observe that there exists a constant $C_1$ such that
\begin{equation}\label{R11}
k_{\tau(s)\theta}(x,y)\leq\exp\{C_{1}|x\|y|\}\quad\forall x,y\in\BR^{2},\,\forall s\in(0,1).
\end{equation}
Since $\cQ_{s}(x,y)\geq(1-s^{2})(|x|^{2}+|y|^{2})$, because $\langle x, y\rangle<0$ in $R_{1}$, we have that if $s_0$ is sufficiently small
\begin{equation}\label{R12}
-\frac{9}{40\,s}\cQ_{s}(x,y)+C_{1}|x\|y|<0\quad\forall(x,y)\in R_{1},\,\forall t\in(0,s_{0}).
\end{equation}
Together (\ref{R11}) and (\ref{R12}) imply (\ref{E2})  in $R_{1}$.
\par
The proof of (\ref{E2})  in $R_2$ is straightforward, because in this region $\cQ_s(x,y)\ge 0$ and $k_{\tau(s)\theta}(x,y)\leq1$.
\par
Next suppose that $(x,y)$ is in $R_{3}$. Since $\langle x,y\rangle\geq 0$, there exists a constant $C_2$ such that 
\begin{align}
k_{\tau(s)\theta}(x,y)&\le\exp\big(C_{2}\mod{x\land y}\big)\nonumber\\
&=   \exp\big(C_{2}\mod{x}\,\mod{y}\,\mod{\sin\vartheta}\big) \qquad\forall s\in(0,1). \label{R3}
\end{align} 
We claim that there exists $s_0\in(0,1)$ such that 
\begin{equation}\label{claimR3}
-\frac{9}{40\,s}\cQ_{s}(x,y)+C_2\mod{x}\,\mod{y}\,\mod{\sin\vartheta}\leq 0\qquad\forall s\in(0,s_0),
\end{equation}
\par
To prove the claim
first consider the case where $|x|\geq|y|$. Then $\cQ_{s}(x,y)\geq|x-y|^{2}$ and  hence,  since $|x-y|\geq|\sin\vartheta||x|$ and $|x-y|\geq\beta|y|$, 
\begin{align*}
-\frac{9}{40\,s}\cQ_{s}(x,y)+C_{2}|x|\,|y|\,\mod{\sin\vartheta}  \le
\left(-\frac{9}{40\,s}\beta+C_2\right)\,\mod{x}\,\mod{y}\,\mod{\sin\vartheta}
\le\  0,
\end{align*}
provided that $s<\frac{9\beta}{40\,C_2}$.\par
Next consider the case where $|x|<|y|$. In this case we have that $\cQ_{s}(x,y)\geq|x-y|^2-2s|y|^2$. Thus, since $|x|<|y|$ and $|x-y|\geq\beta|y|$, 
$$
-\frac{9}{40\,s}\cQ_{s}(x,y)+C_{2}\mod{x}\,\mod{y}\,\mod{\sin\vartheta}\leq\left(-\frac{9}{40\,s}\beta^{2}+\frac{9}{20}+C_2\right) \,|y|^2\le0
$$
provided that $s<\frac{9\beta^2}{40\,C_2+18}$. \par
Thus (\ref{claimR3}) holds for all $(x,y)$ in $R^3$ with $s_0\le \min\set{\frac{9\beta}{40\,C_2},\frac{9\beta^2}{40\,C_2+18} }$. Together (\ref{R3}) and (\ref{claimR3}) imply (\ref{E2})  in $R_{3}$.\\
\par
The proof of  estimate  (\ref{E2})  in $R_{4}$ is similar. Indeed, first of all (\ref{R3}) holds in $R_4$ because here too $\langle x,y\rangle>0$. Moreover, arguing much as before, one can show that  (\ref{claimR3}) holds also for all $(x,y)$ in $R_4$ with $s_0\le \min\set{\frac{9\delta^2}{40\,C_2},\frac{9\delta^2}{40\,C_2+18} }$. The only difference is that one uses the estimates
$$
\cQ_s(x,y)\ge\mod{x-y}^2
\ge\ (\sin\vartheta)^2\,\mod{x}^2\ge \delta^2\,\mod{x}\,\mod{y}
$$
 when $\mod{x}\ge\mod{y}$ and 
 $$
\cQ_s(x,y)\ge |x-y|^2-2s|y|^2 \ge (\sin\vartheta)^2\,\mod{y}^2 -2s\mod{y^2}\ge (\delta^2-2s)\,\mod{y}^2
$$when $\mod{x}<\mod{y}$.
We omit the details. Notice that, so far, we did not need to impose any restriction on $\delta$,  which therefore could be any number in $(0,1)$. 


It remains to estimate $h_{t}(x,y)$ in $R_{5}$. We observe that since $\tau(s)\sim 2s$ as $s\to0^+$ and  $\langle x, y\rangle\geq 0$ and $x\land y<0$  in $R_{5}$, by (\ref{turna}) there exist $s_0>0$ and two positive constants $c_{0}<2<c_{1}$ such that
\begin{align}\label{approx}
k_{\tau(s)\theta}(x,y)\leq\exp\left\{-c_{0}\frac{\theta^{2}}{4}\,s\,\langle x, y\rangle-c_{1}\frac{\theta}{4}\,x\land y\right\} \qquad\forall s\in(0,s_0).
\end{align}
Moreover, we can choose $c_0$ and $c_1$ as close to $2$ as we want, provided that we choose $s_0$ sufficiently small; in particular, we may take 
\begin{equation}\label{c_1^2/c_0}
c_1^2/c_0<18/5.
\end{equation}

Now we are ready to prove estimate (\ref{E2}) in $R_{5}$. Define
$$
E_{s}(x,y)=-\frac{9}{10}\cQ_{s}(x,y)-c_{0}\theta^{2}\,s^2\,\langle x, y\rangle-c_{1}\theta\,s\,x\land y.
$$
By (\ref{approx})
$$
 \exp\left\{-\frac{9}{40\,s}\cQ_{s}(x,y)\right\}k_{\tau(s)\theta}(x,y)\le \exp\left\{\frac{1}{4\,s}E_s(x,y)\right\}.
$$
Thus, to prove (\ref{E2}) in $R_5$ it is enough to show that 
\begin{equation}\label{Es<0}
E_s(x,y)\le 0 \qquad\forall s\in(0,s_0)\ \forall (x,y)\in R_5.
\end{equation}
provided that $s_0$, $\beta$ and $\delta$ are sufficiently small.

 Observe that 
$$
E_s(x,y)=\lambda(x,y)\,s^2+\mu(x,y)\,s+\nu(x,y),
$$
 where
\begin{align*}
&{}\lambda(x,y)=-\frac{9}{10}|x+y|^{2}-c_{0}\theta^{2}\langle x, y\rangle,\\
&{}\mu(x,y)=\frac{18}{10}(|y|^{2}-|x|^{2})-c_{1}\theta\  x\land y,\\
&{}\nu(x,y)=-\frac{9}{10}|x-y|^{2}.
\end{align*}
{\sloppy
It turns out that, instead of $E_s(x,y)$, it is more convenient to consider the function $\mod{x}^{-1}\mod{y}^{-1}\,E_s(x,y)$  because the latter function depends only on the variables $s$, $X=\mod{x}/\mod{y}$ and $\vartheta=\widehat{xy}$. 
}
Indeed, if we denote by $\Psi$ the function defined by $\Psi(s,x,y)=(s,X,\vartheta)$,
\begin{equation}\label{F}
\mod{x}^{-1}\mod{y}^{-1}\,E_s(x,y)=F\big(\Psi(s,x,y)\big),
\end{equation}
where
\begin{equation}\label{}
F(s,X,\vartheta)=\tilde\lambda(X,\vartheta)s^2+\tilde\mu(X,\vartheta)s+\tilde\nu(X,\vartheta),
\end{equation}
and
\begin{align*}
\tilde\lambda(x,y)=&-\frac{9}{10}\,(X+X^{-1}+2\cos\vartheta )-c_0\theta^2\,\cos\vartheta  \\ 
\tilde\mu(x,y)=&\frac{18}{10}(X^{-1}-X)-c_1\theta \sin\vartheta  \\ 
\tilde\nu(x,y)=&-\frac{9}{10}(X+X^{-1}-2\cos\vartheta ).
\end{align*}
 It is easy to see that $(0,1,0)$ is a critical point of $F$ and the Hessian $\nabla^2 F(0,1,0)$ is definite negative because $c_1^2-\frac{18}{5}c_0<0$ by (\ref{c_1^2/c_0}). Thus $(0,1,0)$ is a local maximum of $F$ and, since $F(0,1,0)=0$ there exists a neighbourhood $U$ of $(0,1,0)$ in which $F$ is $\le0$.  Now, since 
 $$
 \Psi\big((0,s_0)\times R_5\big)\subset\set{(s,X,\vartheta): s\in(0,s_0),\ \mod{X-1}<\beta,\ -\delta<\sin(\vartheta)\le0},
 $$
 we can choose $s_0$, $\beta$ and $\delta$ so small that $ \Psi\big((0,s_0)\times R_5\big)\subset U$. Hence $F\circ\Psi\le 0$ in $(0,s_0)\times R_5$. Thus (\ref{Es<0}) is satisfied and the proof of the lemma is complete.\end{proof}
To prove the boundedness of the non-truncated maximal operator we need to assume that the one-parameter group $(\e^{tR})_{t\in\BR}$ generated by the skew-adjoint matrix $R$ is periodic. We recall that if $I$ is an interval contained in $\BR_+$ and $P >0$ we denote by $I_P^\sharp$ the set $\cup_{n\in\BN} (I+ nP)$.
\begin{lemma}\label{Iglob} Suppose that the skew-adjoint matrix $R$ generates a one-para\-meter group $(\e^{tR})_{t\in\BR}$ which  is periodic of period $P$. Then there exist an interval $I$ and  a constant $C$ such that for all $s$ in $\tau^{-1}(I_P^\sharp)$ and all $(x,y)$ in $\BR^d\times\BR^d$
$$
h_{\tau(s)}(x,y)\leq Cs^{-\frac{d}{2}}\ \e^{|x|^{2}-\frac{1}{40 s}\cQ_{s}(x,y)}.
$$
\end{lemma}
\begin{proof}
As in the proof of Lemma \ref{0t_0glob} it is enough to prove the inequality for the kernel $h^\Theta_t(x,y)$, with $\Theta=(\vett{\theta}{d})\in\BR^d$, $\theta_j\geq 0$. 
Let $\set{\vett{\theta}{m}}$ be the nonzero components of $\Theta$, i.e. the absolute values of the nonzero eigenvalues of $R$. Denote by $\theta_{\max}$ the maximum of $\set{\vett{\theta}{m}}$.\par
 Fix $\delta=\min\set{\theta_{\max}^{-1}, 1/10}$ and let $\epsilon$ be a small positive constant ($\epsilon\le 1/10$ will do). Define $I=[\delta,(1+\epsilon)\delta]$. For all $\theta \in\set{\vett{\theta}{m}}$ the functions $t\mapsto\cos(\theta t)$ and $t\mapsto\sin(\theta t)$ are periodic of period $P$ and by considering their Taylor expansions at zero it is easy to see that  for all $\theta\in\set{\vett{\theta}{m}}$
\begin{equation}\label{d2d}
c_0\le 1-\cos(\theta t)\le c_2,\qquad  \sin(\theta t)\le c_1\qquad\forall t \in I_P^\sharp,
\end{equation}
where 
\begin{align}\label{c0c1}
c_0=\frac{5}{12}\theta^2\delta^2\qquad c_1= (1+\epsilon)\delta \theta\qquad{\rm and }\qquad c_2=\frac{(1+\epsilon)^2\delta^2\theta^2}{2}.
\end{align}
Arguing as in the proof of Lemma \ref{0t_0glob}, we may reduce matters 
 to showing that if $\theta\in\set{\vett{\theta}{m}}$ then there exists a constant $C$ such that
 \begin{equation}\label{F2}
 \e^{-\frac{9}{40\,s}\cQ_{s}(x,y)}k_{\tau(s)\theta}(x,y)\le C\qquad\forall s\in \tau^{-1}(I_P^\sharp)\quad \forall (x,y)\in \BR^2\times \BR^2.
\end{equation}
For the sake of the reader we recall that
\begin{align}\label{turna2}
k_{\tau(s)}(x,y)&=\left\{\e^{-\frac{\e^{-t}}{1-\e^{-2t}}\big[\big(1-\cos 
(t\theta)\big)\langle x, y\rangle+\sin (t\theta)\,x\land y\big]}\right\}_{t=\tau(s)} \\ 
&=\e^{-\frac{1-s^2}{4s}[(1-\cos(\tau(s)\theta))\langle x,y\rangle+\sin(\tau(s)\theta)\,x\land y]}.
\end{align}
The set $\BR^2\times\BR^2$ is the disjoint union of the three sets 
\begin{align*}
R_{1}=&\{(x,y)\in \BR^{2}\times\BR^{2}:\,\langle x,y\rangle \geq0,\,x\land y\geq0\}\\
R_{2}=&\{(x,y)\in \BR^{2}\times\BR^{2}:\,\langle x,y\rangle\geq0,\,x\land y<0\},\\
R_{3}=&\{(x,y)\in \BR^{2}\times\BR^{2}:\,\langle x,y\rangle<0\}.
\end{align*}
We shall prove that (\ref{F2}) holds in each region $R_j$, $j=1,2,3$. \par
To prove (\ref{F2})  in $R_{1}$ it is enough to observe that here $k_{t\theta}(x,y)\leq 1$ for all $t$ in $\BR_+$.
\par
 Now suppose that $(x,y)$ is in $R_{2}$. Then, by (\ref{d2d}) and (\ref{turna2}) we have that  
 \begin{align}\label{Fs}
k_{\tau(s)\theta}(x,y)\le \e^{-\frac{1-s^2}{4s}\big(c_0\langle x,y\rangle+c_1\,x\land y\big)}\qquad\forall s\in\tau^{-1}(I_P^\sharp).
\end{align}
Thus
\begin{align*}
 \exp\left\{-\frac{9}{40\,s}\cQ_{s}(x,y)\right\}\ k_{\tau(s)\theta}(x,y)\le \exp\left\{\frac{1}{4\,s}\,F_s(x,y)\right\},
\end{align*}
where
\begin{align*}
F_{s}(x,y)&=p(x,y)\,s^2+q(x,y)\,s+r(x,y)\\
p(x,y)&=-\frac{9}{10}|x+y|^{2}+c_0 \langle x,y\rangle+c_1 \,x\land y,\\
q(x,y)&=\frac{18}{10}(|y|^{2}-|x|^{2}),\\
r(x,y)&=-\frac{9}{10}|x-y|^{2}-c_0\langle x,y\rangle-c_1\,x\land y,.
\end{align*}
Thus, to prove (\ref{F2}) in $R_2$, we only need to show that $F_s(x,y)\le 0$ for all $(x,y)\in R_2$.\par 
It is an easy matter to see that, with $c_0$ and $c_1$ as in (\ref{c0c1}), the leading coefficient $p(x,y)$ and the constant term $r(x,y)$ are negative for all $(x,y)$ in $R_2$. Thus it suffices to show that the discriminant $q^2-4pr$ is nonpositive in $R_2$. 
If $\mod{y}=\mod{x}$ this is obvious, because then $q(x,y)=0$.  If $\mod{y}\not=\mod{x}$, after some simple algebra using the identity
\begin{align*}
|x+y|^{2}|x-y|^{2}=(|y|^{2}-|x|^{2})^{2}+4\sin^2(\vartheta)\,|x|^{2}|y|^{2},
\end{align*}
we see that 
$
(q^{2}-4pr)\,|x|^{-2}|y|^{-2}
$
is only a function of the angle $\vartheta$ between $x$ and $y$. Thus  its sign does not change if we rescale in $x$. In particular we may reduce matters to the case $\mod{y}=\mod{x}$, where $q=0$. This proves that $F_{s}(x,y)\le 0$ for all $(x,y)$ in $R_2$ and $s$ in $\BR$. By  (\ref{Fs}) this implies that (\ref{F2}) holds in $R_2$ .\par\smallskip


Finally suppose that $(x,y)$ is in $R_{3}$. We have that
\begin{align*}
 \exp\left\{-\frac{9}{40\,s}\cQ_{s}(x,y)\right\}\ k_{\tau(s)\theta}(x,y)\le \exp\left\{\frac{1}{4\,s}\,G_s(x,y)\right\},
 \end{align*}
 where
\begin{align*}
G_{s}(x,y)&=\tilde{p}(x,y)\,{s^2}+{q}(x,y)\,s+\tilde{r}(x,y),\\
\tilde{p}(x,y)&=-\frac{9}{10}|x+y|^{2}-c_2\mod{\langle x,y\rangle}+c_1\,\mod{x\land y},\\
q(x,y)&=\frac{18}{10}(|y|^{2}-|x|^{2}),\\
\tilde{r}(x,y)&=-\frac{9}{10}|x-y|^{2}+c_2\mod{\langle x,y\rangle}-c_1\,\mod{x\land y},
\end{align*}
and $c_1$, $c_2$ are as in (\ref{c0c1}). Thus to prove the desired inequality (\ref{F2}), we only need to show that $G_s(x,y)\le 0$ in $R_3$.
Since it is easy to see that  both $\tilde{p}$ and $\tilde{r}$ are negative in $R_{3}$, as before we only need to prove that $q^2-4\tilde{p}\tilde{r}\leq0$ in $R_{3}$. This can be proved by an argument similar to that used  in $R_2$. We omit the details.  \par
Hence (\ref{F2}) holds for all $(x,y)$ in $\BR^2\times\BR^2$. 
This concludes the proof of the lemma. \par
\end{proof}
We recall  two lemmas from \cite{GMMST}.
\begin{lemma}\label{lemma 4.1} Let $\vartheta=\vartheta(x,y)$ denote the angle between the non-zero vectors $x$ and $y$.There exists a constant $C$ such that for all $(x,y)$ in the global region $G$ 
$$
\sup_{0<s\le1}\,s^{-d/2} \,\e^{-\frac{1}{40\,s}\cQ_s(x,y)}\le C\, \min\set{(1+\mod{x})^d, (\mod{x}\sin\vartheta)^{-d}}
$$
\end{lemma}
\begin{lemma}\label{lemma 4.4}
The operator 
$$
\cT f(x)=\e^{\mod{x}^2}\ \int\min\set{(1+\mod{x})^d, (\mod{x}\sin\vartheta)^{-d}}\,f(y)\wrt\gamma_\infty(y)
$$
is of weak type $1$.
\end{lemma}
We are now ready to conclude the proof of Theorem \ref{wt1}
\begin{proof}
Let $A$ denote either the set $[0,t_0]$ or $I_P^\sharp$. By Proposition \ref{weak1loc} the local part of the operator $\cH_{*,A}$ is of weak type $1$. Thus it remains only to prove that the global part is of weak type $1$. By (\ref{HstarG}), Lemma \ref{0t_0glob} ,  Lemma \ref{Iglob} and Lemma \ref{lemma 4.1} the global part of the operator $\cH_{*,A}$ is controlled by the operator $\cT$, which is of weak type $1$ by Lemma \ref{lemma 4.4}.
The conclusion follows by Lemma \ref{RIDUXMAXOP}.
\end{proof}

{\footnotesize
\centerline{\rule{9pc}{.01in}}
\bigskip

\medskip

\medskip
\centerline{Dipartimento di Matematica,
            Universit\`a di Genova,
             Via Dodecaneso 35}
\centerline{16146 Genova, Italy}
\centerline{e-mail: mauceri@dima.unige.it}

}


\begin{thebibliography}{40}

\bibitem{C} Curtis, M. L.  (1984). {Matrix Groups}, Springer-Verlag, Berlin.
\bibitem{DZ} Da Prato, G. and Zabczyk, J.  (1992). {Stochastic Equations 
in Infinite dimensions}, {Cambridge University Press}, Cambridge.
\bibitem{G} Garsia, A. (1970). {Topics in almost everywhere convergence}, Lectures in Advanced Mathematics \textbf{4}, Markham, Chicago.
\bibitem{GMMST} Garc\'ia-Cuerva, J., Mauceri, G., Meda, S., Sj\"ogren, P.
and Torrea, J. L. (2003). {Maximal operators for the holomorphic 
Ornstein-Uhlenbeck semigroup}, \textit{ J. London Math. Soc.(2)} \textbf{ 61}, 
219--234.
\bibitem{MN2} Mauceri and G.,  Noselli, L.   to appear. Riesz transforms for a nonsymmetric Ornstein--Uhlenbeck semigroup, \textit{Semigroup Forum}.
\bibitem{MPS} Men\`arguez, T., P\'erez, S. and Soria, F.  (2000). {The 
Mehler maximal function: a geometric proof of the weak type 1}, \textit{ J. 
London Math. Soc. (2)} \textbf{ 62}, 846--856.
\bibitem{MPP} Metafune, G. , Pallara,  D. and Priola, E.  (2002). {Spectrum 
of Ornstein--Uhlenbeck operators in $L^{p}$ spaces with respect to 
invariant measures}, \textit{ J. Funct. Anal.} \textbf{196}, 40--60.
\bibitem{MPRS} Metafune, G., Pr\"uss, J.,  Rhandi, A.  and Schnaubelt, R.   (2002). 
{The domain of the Ornstein--Uhlenbeck operator on an  
$L^{p}$--space with invariant measure},  \textit{ Ann. Scuola Norm. Sup.
Pisa Cl. Sci.} (5) \textbf{1}, 471--485.
\bibitem{M} Muckenhoupt, B.  (1969). {Poisson integrals for Hermite and Laguerre expansions}, \textit{ Trans. Amer. Math. Soc.} \textbf{ 139}, 231--242.
\bibitem{Sj} Sj\"ogren, P. (1983). {On the maximal function for the 
Mehler kernel}, in Harmonic Analysis, Cortona 1982, (G. Mauceri and G. Weiss, eds.),\textit{ Springer Lecture 
Notes in Mathematics} \textbf{ 992}, 73--82. 
\bibitem{S} Stein, E.M. (1970). {Topics in Harmonic Analysis Related to Littlewood--Paley Theory}, Ann. Math. Studies, 63, Princeton Univ. Press, Princeton, N.J.


\end{thebibliography}
\end{document}